\documentclass[12pt]{article}
\usepackage{amsfonts}
\usepackage{amsmath}
\usepackage{mathrsfs}
\usepackage{color}
\usepackage{tikz}
\usepackage{authblk}
\usepackage{mathrsfs,amscd,amssymb,amsthm,amsmath,bm,graphicx,psfrag,subfigure,url}

\allowdisplaybreaks
\usepackage{indentfirst}

\def \[{\begin{equation}}
\def \]{\end{equation}}

\newtheorem{thm}{Theorem}[section]
\newtheorem{prop}{Proposition}

\newtheorem{lem}[thm]{Lemma}
\newtheorem{cor}[thm]{Corollary}

\newenvironment{wst}
{\setlength{\leftmargini}{1.5\parindent}
 \begin{itemize}
 \setlength{\itemsep}{-1.1mm}}
{\end{itemize}}

\begin{document}
\title{\bf The list-coloring function of signed graphs}
\author[a]{Sumin Huang\thanks{Email: sumin2019@sina.com (S.M.~Huang)}}
\author[a]{Jianguo Qian\thanks{Corresponding author, email: jgqian@xmu.edu.cn (J.G,~Qian)}}
\author[b]{Wei Wang}
\affil[a]{School of Mathematical Sciences, Xiamen
University, Xiamen 361005, P.R. China}
\affil[b]{School of Mathematics, Physics and Finance, Anhui Polytechnic University, \authorcr \it Wuhu 241000, P.R. China}
\date{}
\maketitle

\begin{abstract}
It is known that, for any $k$-list assignment $L$ of a graph $G$,  the number of $L$-list colorings of $G$ is at least the number of the proper $k$-colorings of $G$ when  $k>(m-1)/\ln(1+\sqrt{2})$.  In this paper, we extend the Whitney's broken cycle theorem to $L$-colorings of signed graphs, by which we show that if $k> \binom{m}{3}+\binom{m}{4}+m-1$ then, for any $k$-assignment $L$, the number of $L$-colorings of a signed graph $\Sigma$ with $m$ edges is at least the number of the proper $k$-colorings of $\Sigma$. Further, if $L$ is $0$-free (resp., $0$-included) and  $k$ is even (resp., odd), then the lower bound $\binom{m}{3}+\binom{m}{4}+m-1$ for $k$ can be improved to $(m-1)/\ln(1+\sqrt{2})$.
\end{abstract}

\vspace{2mm} \noindent{\bf Keywords}: signed graph; list-coloring function; broken cycle theorem; chromatic polynomial
\vspace{2mm}

\setcounter{section}{0}
\section{Introduction}\setcounter{equation}{0}

A {\it signed graph} $\Sigma=(G,\sigma)$ consists of an underlying graph $G$ together with a sign function $\sigma:E(G)\to \{-1,1\}$.  An edge $e$ is \textit{negative} if $\sigma(e)=-1$ and is \textit{positive} otherwise. We call a cycle in $\Sigma$ {\it unbalanced} or {\it balanced} if it has an odd or even number of negative edges, respectively. Further, we say that a signed graph is {\it unbalanced} if it contains an unbalanced cycle and is {\it balanced} otherwise.

In 1982, Zaslavsky \cite{Zaslavsky} introduced the notion of the (signed) coloring for signed graphs. For a positive integer $k$,  denote $M_k=\{0, \pm 1, \pm 2, \ldots, \pm t\}$ if $k=2t+1$ and $M_k=\{\pm 1, \pm 2, \ldots, \pm t\}$ if $k=2t$. A {\it proper $k$-coloring}, or {\it $k$-coloring} for short, of a signed graph $\Sigma$ is a mapping $c$ from $V(\Sigma)$ to $M_k$ such that $c(u)\neq \sigma(e)c(v)$ for each edge $e=uv$. Let $P(\Sigma, k)$ denote the number of proper $k$-colorings of  $\Sigma$. By the definition, one can see that the color $0$ plays a very special role since it is self-inverse.  Hence, in contrast to ordinary graphs (i.e., unsigned graphs),  the number of proper $k$-colorings of a signed graph is not a polynomial in general. Indeed, Zaslavsky \cite{Zaslavsky} proved that $P(\Sigma, k)$ is a quasi-polynomial of period two, that is, $P(\Sigma, 2t+1)$ and $P(\Sigma, 2t)$ are both polynomials in $t$. For convenience, we write  $P(\Sigma, k)$ specifically by $P^1(\Sigma, k)$ when $k$ is odd and by $P^0(\Sigma, k)$ when $k$ is even.

For unsigned graphs, the notion of list-coloring was introduced by Erd\H{o}s, Rubin and Taylor in \cite{Erdos}, which can be naturally extended to an analog for signed graphs.  A  {\it list-assignment} $L$ of a signed graph $\Sigma$ is a mapping from every vertex  $v$ of $\Sigma$ to a nonempty set $L(v)$ of permissible colors in $\mathbb{Z}$. For a positive integer $k$, if $|L(v)|=k$ for every $v\in V(\Sigma)$, then we call $L$ a $k$-{\it assignment}. An $L$-{\it coloring} of $\Sigma$ is a proper coloring $c$ such that $c(v)\in L(v)$ for all $v\in V(\Sigma)$. Let $P(\Sigma, L)$ denote the number of $L$-colorings of $\Sigma$. The {\it list-coloring function} of $\Sigma$, denoted by  $P_l(\Sigma, k)$, is the minimum  number of $L$-colorings over all  $k$-assignments $L$, i.e., $P_l(\Sigma,k)=\min \{P(\Sigma, L): L\text{ is a $k$-assignment}\}$.
Let $L^*_k$ be the $k$-assignment such that $L^*_k(v)=M_k$ for every vertex $v$ in $\Sigma$. Hence, $P_l(\Sigma,k)\leq P(\Sigma,L^*_k)=P(\Sigma,k)$. Moreover, we note that the list-coloring of signed graphs is a natural extension of that of unsigned graphs since a signed graph $\Sigma=(G,\sigma)$ without negative edge can be viewed as the unsigned graph $G$.

For an unsigned graph $G$,  Kostochka and Sidorenko \cite{Kostochka} showed that if $G$ is a chordal graph then $P_l(G,k)=P(G,k)$ for every positive integer $k$. This leads to  a nature and interesting question: When does $P_l(G,k)$ equal $P(G,k)$? In 1992, Donner  \cite{Donner} showed that $P_l(G,k)=P(G,k)$ when $k$ is sufficiently large (compared with the
number of vertices). Later in 2009, Thomassen  \cite{Thomassen} specified the `sufficiently large $k$' by
 `$k>n^{10}$', which  was improved further to  `$k>(m-1)/\ln (1+\sqrt{2})$' by two of the present
authors \cite{Wang}, where $n$ and $m$ are the numbers of the vertices and edges in $G$, respectively.

 In contrast to  unsigned graphs, there seems to be very few results on the number of the list colorings for  signed graphs.  In this paper, by extending the Whitney's broken cycle theorem to signed graphs, we prove the following result:

\begin{thm}\label{th0}
Let $\Sigma$ be a signed graph with $m$ edges. If $k> \binom{m}{3}+\binom{m}{4}+m-1$, then
$$P_l(\Sigma,k)=P(\Sigma,k).$$
Further, the list-assignment $L$ that attains $P(\Sigma,k)$ satisfies $L(u)=\sigma(e)L(v)$ for each edge $e=uv$. In particular, if $\Sigma$ is unbalanced, then $L(v)=-L(v)=\{-a:a\in L(v)\}$ for each vertex $v$ in any unbalanced component of $\Sigma$.
\end{thm}

Further, we show that the lower bound `$\binom{m}{3}+\binom{m}{4}+m-1$' of $k$ in Theorem \ref{th0} can be improved for two particular types of list assignments, which give a partial extension of the corresponding result on unsigned graphs \cite{Wang} to signed graphs.  A list-assignment $L$ of $\Sigma$ is $0$-{\it free} if $0\notin L(v)$ for any $v\in V(\Sigma)$, and is $0$-{\it included} if $0\in L(v)$ for any $v\in V(\Sigma)$. Let $P_l^0(\Sigma,k)=\min \{P(\Sigma, L): L\text{ is a $0$-free $k$-assignment}\}$ and $P_l^1(\Sigma,k)=\min \{P(\Sigma, L): L\text{ is a $0$-included $k$-assignment}\}$. Since $0\notin M_{2t}$ and $0\in M_{2t+1}$, we have $P_l^0(\Sigma, 2t)\leq P^0(\Sigma, 2t)$ and $P_l^1(\Sigma, 2t+1)\leq P^1(\Sigma, 2t+1)$.

\begin{thm}\label{th1}
Let $\Sigma$ be a signed graph with $m$ edges. If
$$k> \frac{m-1}{\ln (1+\sqrt{2})},$$
then $P_l^0(\Sigma,k)=P^0(\Sigma,k)$ if $k$ is even and $P_l^1(\Sigma,k)=P^1(\Sigma,k)$ if $k$ is odd. Further, the list-assignment $L$ that attains the equality satisfies $L(v)=\sigma(e)L(u)$ for each edge $e=uv$. In particular, if $\Sigma$ is unbalanced, then $L(v)=-L(v)$ for each vertex $v$ in any unbalanced component of $\Sigma$.
\end{thm}

\section{The broken cycle theorem for $P(\Sigma, L)$}
Let $\Sigma$ be a signed graph and $L$ a list assignment of $\Sigma$. In this section, we will extend  the Whitney's broken cycle theorem to $P(\Sigma, L)$.

 For  $X\subset V(\Sigma)$, the  {\it switching} of $\Sigma$  at $X$ is the operation that reverses the sign of each edge  between $X$ and $V(\Sigma)-X$ (i.e., with one vertex in $X$ and the other in $V(\Sigma)-X$). It is clear that switching defines an equivalence relation on the set of all signed graphs on $G$. We say that two signed graphs  are {\it equivalent} if they can be obtained from each other by a switching at a vertex subset. The following result is a convenient tool to recognize a balanced graph.

\begin{thm}\cite{Zaslavsky2}\label{th2}
For any connected signed graph $(G,\sigma)$, the following assertions are equivalent.
\begin{wst}
\item[\rm (i).] $(G,\sigma)$ is balanced.
\item[\rm (ii).] $(G,\sigma)$ is equivalent to $(G,1)$, that is a signed graph without negative edges.
\item[\rm (iii).] There is a unique partition $X_1\cup X_2$ of $V(G)$  such that the edges between $X_1$ and $X_2$ are exactly all the negative edges of $(G,\sigma)$.
\end{wst}
\end{thm}

For $X\subset V(\Sigma)$  and a list-assignment $L$ of $\Sigma$, the {\it switching} of $L$ at $X$ is the operation that reverses the signs of all colors in $L(v)$ for each $v\in X$. It is clear that if $L'$ is the assignment of $\Sigma$ obtained from $L$ by the switching at $X$ then $L'(v)=-L(v)$ if $v \in X$ and $L'(v)=L(v)$ otherwise. Further, if  $\Sigma'$ is the signed graph obtained from $\Sigma$ by the switching at $X$, then an $L$-coloring $c$ of $\Sigma$ induces naturally  an $L'$-coloring $c'$ of $\Sigma'$, where $c'(v)=-c(v)$ if $v\in X$ and $c'(v)=c(v)$ if $v\notin X$. Hence, $P(\Sigma, L)=P(\Sigma', L')$. Moreover, $L'$ is $0$-free (resp., $0$-included) if and only if $L$ is $0$-free (resp., $0$-included) and, hence, $P^0(\Sigma, L)=P^0(\Sigma', L')$ and $P^1(\Sigma, L)=P^1(\Sigma', L')$. Further, since $L'$ is a $k$-assignment if and only if $L$ is a $k$-assignment, $P_l(\Sigma, k)$, $P_l^0(\Sigma,k)$ and $P_l^1(\Sigma, k)$ are invariant under any switching of $\Sigma$.

Let $T$ be a balanced component of $\Sigma$. By Theorem \ref{th2},  $V(T)$ has a unique partition $X_1\cup X_2$ such that the edges between $X_1$ and $X_2$ are exactly all the negative edges of $T$.  Let $L_1$ and $L_2$ be the list assignments of $\Sigma$ obtained from $L$ by the switchings at $X_1$ and $X_2$, respectively.
\begin{prop}\label{3.1}
$$\left|\bigcap_{v\in V(T)} L_1(v)\right| =\left|\bigcap_{v\in V(T)} L_2(v)\right|.$$
\end{prop}
\begin{proof}
Since $V(T)=X_1\cup X_2$, we have
\begin{align*}
\bigcap_{v\in V(T)} L_1(v)&=\left (\bigcap_{v\in X_1} -L(v) \right )\bigcap \left (\bigcap_{v\in V(T)-X_1} L(v) \right )\\
&=\left (\bigcap_{v\in V(T)-X_2} -L(v) \right )\bigcap \left (\bigcap_{v\in X_2} L(v) \right )\\
&=\bigcap_{v\in V(T)} -L_2(v).
\end{align*}
Note that $c\in \bigcap_{v\in V(T)} L_2(v)$ if and only if $-c \in \bigcap_{v\in V(T)} -L_2(v)$. So $\left|\bigcap_{v\in V(T)} L_1(v)\right| =\left|\bigcap_{v\in V(T)} L_2(v)\right|$, as desired.
\end{proof}

By Proposition \ref{3.1}, we write $\beta(T, L)=\left|\bigcap_{v\in V(T)} L_1(v)\right|=\left|\bigcap_{v\in V(T)} L_2(v)\right|$. Further,  we define $\gamma(\Sigma, L)$ as follows: $\gamma(\Sigma, L)=0$ if there is a vertex $v$ in an unbalanced component of $\Sigma$ such that $0\notin L(v)$, or $\gamma(\Sigma, L)=1$  otherwise. For $F\subseteq E(\Sigma)$, let $\langle F\rangle$ denote the spanning subgraph of $\Sigma$ with edge set $F$ and $b(F)$ the number of the balanced components of $\langle F\rangle$. Under these notations, $P(\Sigma, L)$ can be represented in the following inclusion-exclusion form.
\begin{lem}\label{lembcf}
$$P(\Sigma, L)=\sum_{F\subseteq E(\Sigma)} (-1)^{|F|} \gamma(\langle F\rangle, L) \prod_{j=1}^{b(F)}\beta(T_j,L),$$
where $T_1,T_2,\ldots, T_{b(F)}$ are all the balanced components of $\langle F\rangle$.
\end{lem}
\begin{proof}
For an edge $e=uv$ in $\Sigma$, let $A_e$ be the set of the non-proper vertex colorings $c$ of $\Sigma$ such that $c(w)\in L(w)$ for any $w\in V(\Sigma)$ and  $c(u)=\sigma(e)c(v)$. Then $P(\Sigma, L)$ equals the number of those vertex colorings that are not in $A_e$ for any $e\in E(\Sigma)$. So by the inclusion-exclusion principle, we have
\begin{equation}\label{prod}
P(\Sigma,L)=\sum_{F\subseteq E(\Sigma)} (-1)^{|F|} |\bigcap_{e\in F}A_e|=\sum_{F\subseteq E(\Sigma)} (-1)^{|F|} \prod_{T}\zeta(T,L),
\end{equation}
where the product is over all components $T$ of $\langle F\rangle$ and $\zeta(T,L)$ is the number of the non-proper vertex colorings $c$ of $T$ such that $c(w)\in L(w)$ for any $w\in V(T)$ and  $c(u)=\sigma(e)c(v)$ for each edge $uv\in E(T)$. Hence, it suffices to show that $\prod_{T}\zeta(T,L)= \gamma(\langle F\rangle, L)\prod_{j=1}^{b(F)}\beta(T_j,L)$ for any $F\subseteq E(\Sigma)$.

For a balanced component $T_i$ of $\langle F\rangle$, by Theorem \ref{th2},  $V(T_i)$ has a unique partition $X_1\cup X_2$ such that the edges between $X_1$ and $X_2$ are exactly all the negative edges of $T_i$. Let $T'_i$ and $L'$ be the signed graph and the list assignment obtained from $T_i$ and $L$ by switching at $X_1$, respectively. Notice that $T_i'$ has no negative edge and, hence, can be viewed as an unsigned graph. Therefore, the requirement `$c(u)=\sigma(e)c(v)$ for each $uv\in E(T_i)$' is equivalent to  `$c(u)=c(v)$ for each $uv\in E(T'_i)$', meaning that the colors of all vertices in $T_i'$ are the same. The number of such colorings is clearly equal to $|\bigcap_{v\in V(T'_i)} L'(v)|=\beta(T_i,L)$. Further, notice that, for any edge $e=uv$, a coloring $c$ satisfies $c(u)=\sigma(e)c(v)$ in $T_i$ if and only if $c(u)=c(v)$ in $T'_i$. This means that $\zeta(T_i,L)=\beta(T_i,L)$,   as desired.

For an unbalanced component $T$ of $\langle F\rangle$, we notice that $T$ contains an unbalanced cycle $C$. So by the definition of $A_e$, for any vertex coloring $c$ in $\bigcap_{e\in E(C)}A_e$, the color of every vertex on the cycle $C$  must be $0$ and, hence, the color of every vertex in $T$ must be $0$ as $T$ is connected. The number of such colorings in $T$ is clearly equal to 1 if $0\in\bigcap_{v\in T}L(v)$ or 0 if $0\notin\bigcap_{v\in T}L(v)$. This means that $\prod_{T}\zeta(T,L)= \gamma(\langle F\rangle, L)$, where  the product is over all unbalanced components $T$ of $\langle F\rangle$, again as desired.
\end{proof}

To extend Whitney's broken cycle theorem to $L$-colorings, we introduce the following result given by Dohmen and Trink.
\begin{lem}\label{lembc}\cite{Dohmen}
Let $P$ be a finite linearly ordered set, $\mathscr{B}\subseteq 2^P\backslash {\emptyset}$ and $\Gamma$ be an Abelian group. If $f$ is a mapping from $2^P$ to $\Gamma$ such that, for any $B\in \mathscr{B}$ and $A\supseteq B$,
\begin{align}\label{eqbc}
f(A)=f(A\backslash \{B_{\max}\}),
\end{align}
then
$$ \sum_{A \in 2^P}(-1)^{|A|}f(A)=\sum_{A\in 2^P\backslash \mathscr{B}^{-}} (-1)^{|A|} f(A),$$
where $B_{\max}$ is the maximum element in $B$ and $\mathscr{B}^{-}=\{A:A\in 2^P\text{ and } A\supseteq B\backslash\{B_{\max}\} \text{ for }$ $\text{some }B\in \mathscr{B}\}$.
\end{lem}

In the following, we will apply  Lemma  \ref{lembc} to  $P(\Sigma,L)$. A {\it barbell} in a signed graph $\Sigma$ is the union of two unbalanced cycles $C_1$, $C_2$ and a (possibly trivial) path $P$ with end vertex $v_1\in V(C_1)$ and $v_2\in V(C_2)$, such that $C_1-v_1$ is disjoint from $P\cup C_2$ and $C_2-v_2$ is disjoint from $P\cup C_1$. A {\it circuit} in  $\Sigma$ is either a balanced cycle or a barbell.
Given a linear order on $E(\Sigma)$,  in Lemma \ref{lembc}, we specify  the set `$P$' by $E(\Sigma)$, `$\mathscr{B}$' by the set of circuits of $\Sigma$ and, for any $F\subseteq E(\Sigma)$, let $f(F)=\gamma(\langle F\rangle, L) \prod_{j=1}^{b(F)}\beta(T_j,L)$. It is clear that if $\langle F\rangle$ contains an element of $\mathscr{B}$, then $\langle F\rangle$ contains a balanced cycle or a barbell. Now we show \eqref{eqbc} holds for $f(F)$.

Assume that $\langle F\rangle$ contains a balanced cycle $C$ with the maximum edge $C_{\max}$. Without loss of generality, let $T_1$ be the component of $\langle F\rangle$ containing $C$. If $T_1$ is balanced, then $T_1-C_{\max}$ is balanced. Since  $C$ is a cycle and $C_{\max}$ is an edge on $C$, we have $\beta(T_1,L)=\beta(T_1-C_{\max},L)$. Moreover, balanced components have no effect on $\gamma(\langle F\rangle,L)$, which implies $\gamma(\langle F\rangle,L)=\gamma(\langle F\rangle-C_{\max},L)$. Hence,
$$ \gamma(\langle F\rangle, L)\prod_{j=1}^{b(F)}\beta(T_j,L)=\gamma(\langle F\rangle-C_{\max}, L)\beta(T_1-C_{\max},L)\left( \prod_{j=2}^{b(F)}\beta(T_j,L) \right),$$
that is, $f(F)=f(F\backslash\{C_{\max}\})$, as desired.
Further, we show that if  $T_1$ is unbalanced, then $T_1-C_{\max}$ must be unbalanced. In fact, let $C'$ be an unbalanced cycle of $T_1$. If $C_{\max}\notin E(C')$, then $T_1-C_{\max}$ is unbalanced. If $C_{\max}\in E(C')$,  then  we can verify that $C\cup C'-C_{\max}$ contains an unbalanced cycle. Since deleting edges will not change the list-assignment and each vertex is in an unbalanced component of $\langle F\rangle-C_{\max }$ if and only if it is in an unbalanced component of $\langle F\rangle$, we have $\gamma(\langle F\rangle,L)=\gamma(\langle F\rangle-C_{\max},L)$. So $f(F)=f(F\backslash\{C_{\max}\})$, as desired.

By the discussion above, \eqref{eqbc} is satisfied by any balanced cycle $C\in \mathscr{B}$ and $F\supseteq C$.

Assume now $F$ contains a barbell $B$. Let $T$ be the unbalanced component containing $B$. Notice that $B$ contains two unbalanced cycles. It follows that $T-B_{\max}$ is either an unbalanced component or two  unbalanced components of $\langle F\rangle-B_{\max}$.  Since deleting edges will not change the list-assignment and each vertex is in an unbalanced component of $\langle F\rangle-B_{\max }$ if and only if it is in an unbalanced component of $\langle F\rangle$, we have $\gamma(\langle F\rangle,L)=\gamma(\langle F\rangle-B_{\max},L)$.  So $f(F)=f(F\backslash\{B_{\max}\})$, again as desired. Hence, \eqref{eqbc} is satisfied by any barbell $B\in \mathscr{B}$ and $F\supseteq B$.

Given a linear order on $E(\Sigma)$, a {\it broken circuit} of $\Sigma$ is a set of edges obtained from the edge set of a circuit of $\Sigma$ by removing its maximum edge. Define a set system
$$\mathcal{C}(\Sigma)=\{F:F\subseteq E(\Sigma) \text{ and }  F \text{ contains no broken circuit}\}.$$
By the definition above, for any $F\in\mathcal{C}(\Sigma)$,  each component $T$ of $\langle F\rangle$ is a tree, or contains a unique and unbalanced cycle. Hence, $|E(T)|= |V(T)|-1$ if $T$ is a tree or $|E(T)|= |V(T)|$ if $T$ contains a  unique unbalanced cycle, meaning that $|F|\leq n$ and $b(F)=n-|F|$. So we can write
$$\mathcal{C}(\Sigma)=\mathcal{C}_0(\Sigma)\cup \mathcal{C}_1(\Sigma)\cup\cdots\cup \mathcal{C}_n(\Sigma),$$
where $\mathcal{C}_i(\Sigma)=\{F:F\in \mathcal{C}(\Sigma) \text{ and } |F|=i\}$ for $i=0,1,\ldots,n$. Further, if $\langle F\rangle$ is required to be balanced, then $\mathcal{C}_n(\Sigma)=\emptyset$. Define $\mathcal{C}_i^{*}(\Sigma)=\{F:F\in \mathcal{C}_i(\Sigma) \text{ and } \langle F\rangle \text{ is balanced}\}$ for $i=0,1,\ldots,n-1$.

\begin{thm}\label{th4}
Let $\Sigma$ be a signed graph with a linear order on $E(\Sigma)$ and $L$ a list-assignment of $\Sigma$. Then
\begin{align}\label{eqfl}
P(\Sigma, L)=\sum_{i=0}^{n-1} (-1)^i \sum_{F\in \mathcal{C}_i(\Sigma)}  \gamma(\langle F\rangle, L)\prod_{j=1}^{n-i}\beta(T_j,L)+\sum_{F\in \mathcal{C}_n(\Sigma)}(-1)^n\gamma(\langle F\rangle,L),
\end{align}
where $T_1,T_2,\ldots, T_{n-i}$ are all the balanced components of $\langle F\rangle$.

In particular, if $L$ is $0$-included, that is, $\gamma(\langle F\rangle,L)=1$ for any $F\subseteq E(\Sigma)$, then
\begin{align*}
P(\Sigma, L)=\sum_{i=0}^{n-1} (-1)^i \sum_{F\in \mathcal{C}_i(\Sigma)}\prod_{j=1}^{n-i}\beta(T_j,L)+\sum_{F\in \mathcal{C}_n(\Sigma)}(-1)^n.
\end{align*}

If $L$ is $0$-free, that is $\gamma(\langle F\rangle,L)=0$ for any unbalanced $\langle F\rangle$, then
\begin{align*}
P(\Sigma, L)=\sum_{i=0}^{n-1} (-1)^i \sum_{F\in \mathcal{C}_i^{*}(\Sigma)} \prod_{j=1}^{n-i}\beta(T_j,L).
\end{align*}
\end{thm}
{\bf Remark.} In Theorem \ref{th4}, since $\langle F\rangle$ contains no broken circuit, each component of $\langle F\rangle$ is either a tree, or contains a unique and unbalanced cycle. So each balanced component $T_j$ must be a tree.

In Theorem \ref{th4}, let $L$ be the $k$-assignment such that $L(v)=M_k$ for any vertex $v$ in $\Sigma$. Then $P(\Sigma,L)=P(\Sigma,k)$ and $L$ is $0$-included (resp., $0$-free) if $k$ is odd (resp., even). So we have the following corollary, which was introduced in \cite{Ren}.
\begin{cor}\label{cor}
Let $\Sigma$ be a signed graph with a linear order on $E(\Sigma)$. Then
\begin{align*}
P^1(\Sigma, k)=\sum_{i=0}^{n-1}(-1)^i \sum_{F\in \mathcal{C}_i(\Sigma)} k^{n-i}+\sum_{F\in \mathcal{C}_n(\Sigma)}(-1)^n
\end{align*}
and
\begin{align*}
P^0(\Sigma, k)=\sum_{i=0}^{n-1}(-1)^i \sum_{F\in \mathcal{C}_i^{*}(\Sigma)}k^{n-i}.
\end{align*}
\end{cor}

\section{The proof of Theorem \ref{th0}}
In this section, we apply Theorem \ref{th4} to prove Theorem \ref{th0}. Since the number of $L$-colorings of a signed graph equals the product of  that of its components, in the following we only consider the case that $\Sigma$ is connected. For any positive integer $k$, let $L_k$ be a $k$-assignment that attains $P(\Sigma,L_k)=P_l(\Sigma,k)$ and recall that $L^*_k$ is the $k$-assignment such that $L^*_k(v)=M_k$ for every vertex $v$ in $\Sigma$. Hence, $P(\Sigma,L^*_k)=P(\Sigma,k)$.
%Firstly, we give two observations for each term $\gamma(\langle F\rangle, L_k)\prod_{j=1}^{n-i}\beta(T_j,L_k) $ in \eqref{eqfl}.

%\begin{Obs}
%Assume that $g$ is the minimum length of an unbalanced cycle which contains a vertex $v$ such that $0\notin L(v)$. Then by the definition of $\gamma(\langle F\rangle,L)$, we have $\gamma(\langle F\rangle,L)=1$ if $|F|<g$ or $F\in \mathcal{C}_1(\Sigma)\cup \mathcal{C}_2(\Sigma)$  as $g\geq 3$.
%\end{Obs}

%\begin{Obs}
%$\prod_{j=1}^{n-i}\beta(T_j,L)\leq  k^{n-i}$, which means $\prod_{j=1}^{n-i}\beta(T_j,L)$ {\color{blue}has no contribution to $k^{n-j}$ for any $j<i$.} In particular, $\prod_{j=1}^{n-i}\beta(T_j,M)= k^{n-i}$.
%\end{Obs}

%A list assignment $L$ is called {\it constant} if $L(u)=\sigma(uv) L(v)$ for any edge $uv$ of $\Sigma$.
%\begin{lem}\label{const1} For any positive integer $k$, let $L_k$ be a $k$-assignment that attains $P(\Sigma,L_k)=P_l(\Sigma,k)$. Then there is an integer $N>0$ such that $L_k$ is constant for any $k\geq N$.
%\end{lem}
%\begin{proof}Suppose to the contrary that the lemma is false. Then there is a infinite sequence $k_1,k_2,\ldots,k_i,\ldots$ such that  $L_k$ is not constant for every $k\in\{k_1,k_2,\ldots,k_i,\ldots\}$.

 For each $F\in \mathcal{C}_1(\Sigma)$, $\langle F\rangle$ consists of one edge $e$ and $n-2$ isolated vertices.  For each edge $e=uv\in E(\Sigma)$, let $\alpha(e,L_k)=k-|L_k(u)\cap L_k(v)|$ if $e$ is positive and $\alpha(e,L_k)=k-|(-L_k(u))\cap L_k(v)|$ if $e$ is negative. Then we have
\begin{align}
\sum_{F\in \mathcal{C}_1(\Sigma)} \prod_{j=1}^{n-1}\beta(T_j,L_k)=\sum_{e\in E(\Sigma)} (k-\alpha(e,L_k))k^{n-2}=m k^{n-1}-\sum_{e\in E(\Sigma)} \alpha(e,L_k) k^{n-2}. \label{eqa1}
\end{align}

For each $F\in \mathcal{C}_2(\Sigma)$, $\langle F\rangle$ consists of either a path of length $2$ and $n-3$ isolated vertices, or two independent edges and $n-4$ isolated vertices. Let $\mathcal{P}(\Sigma)$ (resp., $\mathcal{P}'(\Sigma)$) be the set of the paths of length $2$ (resp., two independent edges) in $\Sigma$ and $m_1=|\mathcal{P}(\Sigma)|$ (resp., $m_2=|\mathcal{P}'(\Sigma)|$). Then
\begin{align}
\sum_{F\in \mathcal{C}_2(\Sigma)} \prod_{j=1}^{n-2}\beta(T_j,L_k)
=\sum_{P\in \mathcal{P}(\Sigma)}\beta(P,L_k)k^{n-3}+\sum_{e_1\cup e_2\in \mathcal{P}'(\Sigma)} \beta(e_1,L_k)\beta(e_2,L_k)k^{n-4}.\label{eqa2}
\end{align}
For $P=uvw\in \mathcal{P}(\Sigma)$, we have
\begin{align*}
\beta(P,L_k)=&\ \left | (\sigma(uv) L_k(u)) \cap L_k(v)\cap(\sigma(vw) L_k(w))\right | \\
=&\ \left |(\sigma(uv) L_k(u))\cap L_k(v)\right |+\left |L_k(v) \cap (\sigma(vw) L_k(w))\right |\\
&\ -\left | \left ((\sigma(uv) L_k(u))\cup (\sigma(vw) L_k(w))\right )\cap L_k(v)\right |\\
\geq&\ \left (k-\alpha(uv,L_k)\right )+\left (k-\alpha(vw,L_k)\right )-k\\
=&\ k-(\alpha(uv,L_k)+\alpha(vw,L_k)).
\end{align*}
For $P=e_1\cup e_2\in \mathcal{P}'(\Sigma)$, we have
\begin{align*}
\beta(e_1,L_k)\beta(e_2,L_k)=&\ (k-\alpha(e_1,L_k))(k-\alpha(e_2,L_k))\\
=&\ k^2-\left (\alpha(e_1,L_k)+\alpha(e_2,L_k)\right )k+\alpha(e_1,L_k)\alpha(e_2,L_k)\\
\geq&\ k^2-\left (\alpha(e_1,L_k)+\alpha(e_2,L_k)\right )k.
\end{align*}

Hence, by \eqref{eqa2}, it follows that
\begin{align}
\sum_{F\in \mathcal{C}_2(\Sigma)} \prod_{j=1}^{n-2}\beta(T_j,L_k)
\geq&\ m_1 k^{n-2}-\sum_{uvw \in \mathcal{P}(\Sigma)} \left (\alpha(uv,L_k)+\alpha(vw,L_k)\right )k^{n-3}\notag \\
&\ + m_2 k^{n-2}- \sum_{e_1\cup e_2\in \mathcal{P}'(\Sigma)} \left (\alpha(e_1,L_k)+\alpha(e_2,L_k)\right )k^{n-3}\notag \\
= &\ (m_1+m_2) k^{n-2}-\sum_{e=uv\in E(\Sigma)} \alpha(e,L_k)(d(u)+d(v)-2) k^{n-3}\notag \\
&\ -\sum_{e=uv\in E(\Sigma)} \alpha(e,L_k)(m-d(u)-d(v)+1) k^{n-3}\notag \\
= &\ (m_1+m_2) k^{n-2}- \sum_{e\in E(\Sigma)} \alpha(e,L_k) (m-1) k^{n-3}, \label{eqb2}
\end{align}
where $d(u)$ and $d(v)$ are the degrees of $u$ and $v$, respectively, and the first equality holds because $\alpha(e,L_k)$ occurs $d(u)+d(v)-2$ times in $\sum_{uvw \in \mathcal{P}(\Sigma)} \left (\alpha(uv,L_k)+\alpha(vw,L_k)\right )$ and $m-d(u)-d(v)+1$ times in $\sum_{e_1\cup e_2\in \mathcal{P}'(\Sigma)} \left (\alpha(e_1,L_k)+\alpha(e_2,L_k)\right )$ for each $e\in E(\Sigma)$.

By \eqref{eqfl} and the classical Bonferroni inequality (also known as inclusion-exclusion inequality),
$$P(\Sigma,L_k)\geq \sum_{i=0}^{3} (-1)^i \sum_{F\in \mathcal{C}_i(\Sigma)}  \gamma(\langle F\rangle, L_k)\prod_{j=1}^{b(F)}\beta(T_j,L_k).$$
Note that $\gamma(\langle F\rangle),L_k)=1$ for each $F\in \bigcup_{i=0}^{2}\mathcal{C}_i$. Together with \eqref{eqa1} and \eqref{eqb2}, we have
\begin{align*}
P(\Sigma,L_k)\geq &k^n-m k^{n-1}+\sum_{e\in E(\Sigma)} \alpha(e,L_k) k^{n-2}+(m_1+m_2) k^{n-2}\\
& - \sum_{e\in E(\Sigma)} \alpha(e,L_k) (m-1) k^{n-3}- \sum_{F\in \mathcal{C}_3(\Sigma)}  \gamma(\langle F\rangle, L_k)\prod_{j=1}^{b(F)}\beta(T_j,L_k)\\
\geq & k^n-m k^{n-1}+(m_1+m_2) k^{n-2}+ \sum_{e\in E(\Sigma)} \alpha(e,L_k) (k-m+1) k^{n-3}
-\sum_{F\in \mathcal{C}_3(\Sigma)} k^{n-3}\\
\geq & k^n-m k^{n-1}+(m_1+m_2) k^{n-2}+ \sum_{e\in E(\Sigma)} \alpha(e,L_k) (k-m+1) k^{n-3}- |\mathcal{C}_3| k^{n-3} \\
\geq & k^n-m k^{n-1}+(m_1+m_2) k^{n-2}+ \sum_{e\in E(\Sigma)} \alpha(e,L_k) (k-m+1) k^{n-3}- \binom{m}{3} k^{n-3}.
\end{align*}

On the other hand, again by \eqref{eqfl} and the Bonferroni inequality, we have
\begin{align*}
P(\Sigma,L^*_k)\leq & k^n-m k^{n-1}+(m_1+m_2) k^{n-2}-\sum_{F\in \mathcal{C}_3(\Sigma)} \gamma(\langle F\rangle,L^*_k) k^{n-3}+\sum_{F\in \mathcal{C}_4(\Sigma)} \gamma(\langle F\rangle,L^*_k) k^{n-4}\\
\leq & k^n-m k^{n-1}+(m_1+m_2) k^{n-2}+\sum_{F\in \mathcal{C}_4(\Sigma)} k^{n-4}\\
\leq & k^n-m k^{n-1}+(m_1+m_2) k^{n-2}+\binom{m}{4} k^{n-3}.
\end{align*}

Hence,
\begin{align*}
P(\Sigma, L_k)-P(\Sigma,L^*_k)\geq  (k-m+1) k^{n-3}\sum_{e\in E(\Sigma)} \alpha(e,L_k)-  \left(\binom{m}{3}+\binom{m}{4}\right) k^{n-3}.
\end{align*}

If $\sum_{e\in E(\Sigma)} \alpha(e,L_k) \geq 1$, then  $P(\Sigma, L_k)-P(\Sigma,L^*_k)> 0$ as  $k> \binom{m}{3}+\binom{m}{4}+m-1$. Recall that  $P(\Sigma, L_k)= P_l(\Sigma,k)$ and $P_l(\Sigma, k)\leq P(\Sigma,L^*_k)$. This is a contradiction.

Hence, the only possibility is $\sum_{e\in E(\Sigma)} \alpha(e,L_k)=0$,  meaning that $\alpha(e,L_k)=0$ for each edge $e\in E(\Sigma)$.
So $L_k(u)=\sigma(e)L_k(v)$ for each edge $e=uv$. If $\Sigma$ is unbalanced, then let $C$ be an unbalanced cycle of $\Sigma$. As $C$ has an odd number of negative edges and $\alpha(e,L_k)=0$ for each $e\in E(\Sigma)$, we have $L_k(v)=-L_k(v)$ for any $v\in V(C)$. Further, since $\Sigma$ is connected,  $L_k(v)=-L_k(v)$ for any $v\in V(\Sigma)$. This implies $L_k=L^*_k$ by relabeling the colors of $L_k$. Hence, $P_l(\Sigma,k)=P(\Sigma,L_k)=P(\Sigma,L^*_k)=P(\Sigma,k)$, as desired. This completes our proof of Theorem \ref{th0}.

\section{The proof of Theorem \ref{th1}}
In this section, we give the proof of Theorem \ref{th1} by using Theorem \ref{th4} and Corollary \ref{cor}. Let $\Sigma$ be a signed graph of order $n$ with $m$ edges. By Theorem \ref{th4} and Corollary \ref{cor}, for $0$-included $k$-assignment $L^1$ and $0$-free $k$-assignment $L^0$, we have
\begin{align}\label{eql1}
P(\Sigma, L^1)-P^1(\Sigma, k)=\sum_{i=0}^{n-1}(-1)^{i-1} \sum_{F\in \mathcal{C}_i(\Sigma)} \left(k^{n-i}-\prod_{j=1}^{n-i}\beta(T_j,L^1)\right)
\end{align}
and
\begin{align}\label{eql0}
P(\Sigma, L^0)-P^0(\Sigma, k)=\sum_{i=0}^{n-1} (-1)^{i-1} \sum_{F\in \mathcal{C}_i^{*}(\Sigma)}\left(k^{n-i}-\prod_{j=1}^{n-i}\beta(T_j,L^0)\right),
\end{align}
where $T_1,T_2,\ldots,T_{n-i}$ are all the balanced components of $\langle F\rangle$.  Let $L$ be a $k$-assignment of $\Sigma$. For $F\subseteq E(\Sigma)$ and a balanced component $T_j$ of $\langle F\rangle$, there is a unique partition $X_j\cup (V(T_j)-X_j)$ of $V(T_j)$ such that the edges between $X_j$ and $V(T_j)-X_j$ are exactly all the negative edges in $T_j$ by Theorem \ref{th2}.  Let $T_j'$ be obtained from $T_j$ by the switching at  $X_j$ and $L'$ be the $k$-assignment obtained from $L$ by the switching at $\bigcup_{j=1}^{n-i} X_j$.
%Let $L_j$ be the sub-list assignment of $L$ restricted on $T_j$, and let $T_j'$ and $L'_j$ be obtained from $T_j$ and $L_j$ by the switchings at  $X_j$, respectively.
Then  $T_j'$ has no negative edges and $\beta(T_j, L)=\left|\bigcap_{v\in V(T_j')} L'(v)\right|$. Further, by the Remark of Theorem \ref{th4}, $T_j$ is a tree, and so is $T'_j$. The following inequality for unsigned trees (trees without negative edge) was given by two of the present authors \cite{Wang}, which is also valid for $T_j'$ because $T_j'$ contains no negative edge.

\begin{lem}\label{lemwang} (\cite{Wang}, Lemma 6). For any unsigned trees $T_1', \ldots, T_{n-i}'$,%$F\subseteq E(\Sigma)$,
$$k^{n-i}-\prod_{j=1}^{n-i}\left |\bigcap_{v\in V(T_j')} L'(v)\right | \leq k^{n-i-1}\sum_{j=1}^{n-i} \sum_{uv\in E(T'_j)} \left ( k- |L'(u)\cap L'(v)|\right ).$$
\end{lem}

By the definition of switching of $L$, for $uv\in E(T_j')$, we have
$$
|L'(u)\cap L'(v)|=
\left \{
\begin{aligned}
|L(u)\cap L(v)|,&\quad \text{if } u,v \notin X_j,\\
|(-L(u))\cap L(v)|,&\quad \text{if } u\in X_j \text{ and } v\notin X_j,\\
|L(u)\cap(-L(v))|,&\quad \text{if } u\notin X_j \text{ and } v\in X_j,\\
|(-L(u))\cap(-L(v))|,&\quad \text{if } u,v \in X_j.
\end{aligned}
\right.
$$
For each edge $e=uv\in E(\Sigma)$, recall that $\alpha(e,L)=k-|L(u)\cap L(v)|$ if $e$ is positive and $\alpha(e,L)=k-|(-L(u))\cap L(v)|$ if $e$ is negative. Further, notice that $e=uv$ is negative if and only if exact one of $\{u,v\}$ is in $X_j$. Combining with the fact  $|L(u)\cap L(v)|=|(-L(u))\cap(-L(v))|$ and $|(-L(u))\cap L(v)|=|L(u)\cap(-L(v))|$, one has $k-|L'(u)\cap L'(v)|=\alpha(e,L)$. Hence, by Lemma \ref{lemwang},
\begin{align}\label{eql}
k^{n-i}-\prod_{j=1}^{n-i}\beta(T_j,L) \leq& k^{n-i-1}\sum_{j=1}^{n-i}\sum_{e\in E(T_j)}\alpha(e,L)\notag \\
\leq&k^{n-i-1} \sum_{e\in F}\alpha(e,L).
\end{align}

Let $\mathcal{E}_i=\{F: F\subseteq E(\Sigma)\text{ and }|F|=i\}$ and $\mathcal{E}_i^e=\{F: F\subseteq \mathcal{E}_i\text{ and }e\in F\}$ for each edge $e\in E(\Sigma)$. It is clear that $\mathcal{C}_i\subseteq \mathcal{E}_i$, $\mathcal{C}_i^{*}\subseteq \mathcal{E}_i$ and $|\mathcal{E}_i^e|=\binom{m-1}{i-1}$. Let $\alpha_i^1=\sum_{F\in \mathcal{C}_i(\Sigma)} \left (k^{n-i}-\prod_{j=1}^{n-i}\beta(T_j,L^1) \right )$ and $\alpha_i^{0}= \sum_{F\in \mathcal{C}_i^{*}(\Sigma)}\left(k^{n-i}- \prod_{j=1}^{n-i}\beta(T_j, L^0)\right)$. Clearly, $\alpha_i^1\geq 0$ and $\alpha_i^0\geq 0$.

 According to \eqref{eql}, we have
\begin{align*}
\alpha_i^1&\leq \sum_{F\in \mathcal{C}_i(\Sigma)}\left ( k^{n-i-1}\sum_{e\in E(F)}\alpha(e,L^1)\right )\\
&\leq  k^{n-i-1}\sum_{F\in \mathcal{E}_i(\Sigma)}\sum_{e\in E(F)}\alpha(e, L^1)\\
&= k^{n-i-1}\sum_{e\in E(\Sigma)}\sum_{F\in \mathcal{E}_i^e(\Sigma)} \alpha(e, L^1)\\
&= k^{n-i-1} \binom{m-1}{i-1} \sum_{e\in E(\Sigma)} \alpha(e, L^1).
\end{align*}
By a similar argument, $\alpha_i^{0}\leq k^{n-i-1} \binom{m-1}{i-1} \sum_{e\in E(\Sigma)} \alpha(e, L^0).$ Further, by a direct calculation, it can be found that $\alpha_0^t=0$ and $\alpha_1^t=k^{n-2}\sum_{e\in E(\Sigma)}\alpha(e,{L^t})$, where $t\in\{0,1\}$.

Hence, for $t\in \{0,1\}$, combining with \eqref{eql1} and \eqref{eql0}, we have
\begin{align*}
P(\Sigma, L^t)-P^t(\Sigma,k)&=\sum_{i=0}^{n-1}(-1)^{i-1} \alpha_i^{t}\\
&=\alpha_1^{t}+\sum_{i=2}^{n-1}(-1)^{i-1} \alpha_i^{t}\\
&\geq k^{n-2}\sum_{e\in E(\Sigma)}\alpha(e, L^t)-\sum_{2\leq i\leq n-1\atop i \text{ is even}} k^{n-i-1}\binom{m-1}{i-1}\sum_{e\in E(\Sigma)}\alpha(e, L^t)\\
&= \sum_{e\in E(\Sigma)}\alpha(e, L^t) \left ( k^{n-2}-\sum_{1\leq i\leq n-2\atop i \text{ is odd}}k^{n-i-2}\binom{m-1}{i}\right)\\
&= \sum_{e\in E(\Sigma)}\alpha(e, L^t) k^{n-2}\left ( 1-\sum_{1\leq i\leq n-2\atop i \text{ is odd}}\binom{m-1}{i}k^{-i}\right)\\
&\geq \sum_{e\in E(\Sigma)}\alpha(e, L^t) k^{n-2}\left ( 1-\sum_{1\leq i\leq n-2\atop i \text{ is odd}}\frac{1}{i!}\left (\frac{m-1}{k}\right )^{i}\right)\\
&\geq  \sum_{e\in E(\Sigma)}\alpha(e, L^t) k^{n-2}\left (1-\frac{1}{2}\left( \exp \left( \frac{m-1}{k} \right)-\exp \left(- \frac{m-1}{k} \right) \right) \right).
\end{align*}

Since the function $1-\left( \exp (x)-\exp (-x) \right)/2$ is monotone decreasing with unique zero $\ln (1+\sqrt{2})$, we have $P(\Sigma, L^t)-P^t(\Sigma,k)\geq 0$ when $k> (m-1)/\ln (1+\sqrt{2})$, where the equality holds if and only if $\sum_{e\in E(\Sigma)}\alpha(e, L^t)=0$. If $\sum_{e\in E(\Sigma)}\alpha(e,L^t)\not=0$, i.e., $\sum_{e\in E(\Sigma)} \alpha(e,L^t) \geq 1$, then  $P(\Sigma, L^t)-P(\Sigma,L^*_k)> 0$ when $k> (m-1)/\ln (1+\sqrt{2})$.

We now assume that $\sum_{e\in E(\Sigma)} \alpha(e,L^t)=0$. Then $\alpha(e,L^t)=0$ for each edge $e\in E(\Sigma)$.
So $L^t(u)=\sigma(e)L^t(v)$ for each edge $e=uv$. If $\Sigma$ is unbalanced, then let $C$ be an unbalanced cycle of $\Sigma$.  As $C$ has an odd number of negative edges and $\alpha(e,L^t)=0$ for each $e\in E(\Sigma)$, we have $L^t(v)=-L^t(v)$ for any $v\in V(C)$. Further, since $\Sigma$ is connected,  $L^t(v)=-L^t(v)$ for any $v\in V(\Sigma)$. This implies $L^t=L^*_k$ by relabeling the colors of $L^t$. Hence, $P(\Sigma,L^t)=P(\Sigma,L^*_k)=P(\Sigma,k)$.

 The argument above shows that $P(\Sigma,L^t)\geq P^t(\Sigma,k)$ and the equality holds if and only if $L^t=L^*_k$. This completes our proof of Theorem \ref{th1}.

\section{Acknowledgements}
This work was supported by the National Natural Science Foundation of China [Grant numbers, 11971406, 12171402] and the National Science Fund for Distinguished Young Scholars [Grant numbers, 12001006].

\end{document}